\documentclass[10pt,a4paper,reqno,twoside]{amsart}

\usepackage{verbatim}
\usepackage{amssymb}
\usepackage{upref}
\usepackage{amsmath}
\usepackage{amsthm}
\usepackage{amsfonts}
\usepackage{latexsym}
\usepackage{amssymb}
\usepackage{color}
\usepackage{amsmath}
\usepackage{amsfonts}
\usepackage{amsthm}
\usepackage{amssymb}
\usepackage{amsmath,amsfonts,amscd,bezier}
\usepackage[latin1]{inputenc}

\usepackage{enumerate}

\theoremstyle{plain}
\newtheorem{Thm}{Theorem}
\newtheorem*{Thm*}{Theorem}

\newtheorem{Prop}{Proposition}
\newtheorem{Lem}[Prop]{Lemma}
\newtheorem{Cor}[Prop]{Corollary}
\theoremstyle{definition}
\newtheorem{Def}{Definition}
\newtheorem{Hyp}{Hypothesis}

\theoremstyle{remark}
\newtheorem{Rem}{Remark}

\newtheorem{Example}{Example}

\allowdisplaybreaks

\newcommand{\N}{\mathbb{N}}

\newcommand{\R}{\mathbb{R}}
\newcommand{\Rn}{{\R^n}}

\DeclareMathOperator{\diag}{diag}

\DeclareMathOperator{\supp}{supp}

\renewcommand{\doteq}{:=}

\newcommand{\lloc}{{\mathrm{loc}}}

\newcommand{\expo}{{\kappa}}

\newcommand{\Fuj}{{\,\mathrm{Fuj}}}
\newcommand{\Kato}{{\,\mathrm{Kat}}}
\newcommand{\Sob}{{\,\mathrm{Sob}}}
\newcommand{\Str}{{\,\mathrm{Str}}}

\newcommand{\Crit}{{\mathrm{C}}}

\def\<#1\>{\left\langle#1\right\rangle }

\begin{document}

\baselineskip15pt

\title[A modified test function method for damped wave equations]
{A modified test function method\\for damped wave equations}

\author{Marcello D'Abbicco \ {\it \& } Sandra Lucente}
\address{Dipartimento di Matematica,
Universit\`a degli Studi di Bari,  Via Orabona 4, 70125 Bari, Italy} \email{dabbicco@dm.uniba.it, lucente@dm.uniba.it}

\begin{abstract}
In this paper we use a \emph{modified} test function method to derive nonexistence results for the semilinear wave equation with time-dependent speed and damping. The obtained critical exponent is the same exponent of some recent results on global existence of small data solution.
\end{abstract}

\keywords{semi-linear equations, damped wave equations, critical exponent, test function method}

\subjclass[2010]{35L71, 35D05}
\maketitle


\section{Introduction}\label{sec:intro}

We consider the Cauchy problem for the wave equation with time-dependent speed and damping
\begin{equation}
\label{eq:dissblow}
\begin{cases}
u_{tt}-a(t)\triangle u+b(t)u_t=f(t,x)\,|u|^p, & t\geq0, \ x\in\R^n,\\
u(0,x)=u_0(x), \\
u_t(0,x)=u_1(x).
\end{cases}
\end{equation}
with $a(t)>0, b(t)>0, f(t,x)>0$ and $p>1$.
Recently a big effort has been done for proving global existence for \eqref{eq:dissblow}. In small data context this means to find an exponent $p_\Crit$ such that the local-in-time solution can be extended to a global one, for any $p>p_\Crit$. One may ask if this exponent $p_\Crit$ is \emph{critical}, that is, if for $p<p_\Crit$ a non-existence result can be established.
\\
In the case $a=b=f=1$ Todorova and Yordanov~\cite{TY} proved that~$p_\Crit=1+2/n$ is critical. The nonexistence result for~$p=p_\Crit$ has been established in \cite{zhang}. In the case~$a=f=1$ and for $b(t)=b_0\,(1+t)^\beta$, with~$|\beta|<1$, the exponent~$p_\Crit=1+2/n$ is still critical, see \cite{LNZ}.
Recently global existence of small data solutions for~$p>p_\Crit$ has been proved in \cite{DLR} for a general damping coefficient~$b=b(t)$ sufficiently regular, provided that it is \emph{effective}, that is, the damped wave equation inherits \emph{parabolic} properties, see~\cite{W07}. If~$f=f(t)$ then the global existence holds for $p>\overline{p}$, where $\overline{p}$ depends on the interaction between the asymptotic behaviors of $b(t)$ and $f(t)$, see~\cite{DA12}.
\\
On the other hand, if we have a space-dependent damping $b(x)u_t$ the critical exponent $p_\Crit$ is modified by the decaying behavior of $b(x)$, see \cite{ITY}. For the existence result with damping term depending on time and space variables, we address to \cite{LNZ11, W}.
\\
If we consider global existence of classical solutions for semilinear waves with time-dependent propagation speed $a(t)$, no damping and $f=1$, the range of admissible exponents $p$ for large data depend on $a(t)$, see \cite{D, DD, FL}. If, moreover, $f=f(t)$, this range is modified by the interaction between the order of zeros of $a(t)$ and the order of zeros of $f(t)$, see \cite{Lu}. A global existence result for Schr\"odinger operator with time-dependent coefficients has been obtained in \cite{Fanelli}.
\\
The interaction between the coefficients~$a(t)$ and~$b(t)$ comes into play also if one studies linear decay estimates for the wave equation with time-dependent speed and damping, see~\cite{DAE12, DAE13, R}.
\\
We expect a nonexistence result for  weak solution to \eqref{eq:dissblow} for $p\le \overline{p}$ where $\bar p$ depends on the interaction between $a(t),b(t)$ and $f(t,x)$. We will use the following definition of weak solutions.
\begin{Def}\label{Def:weakSol1} Let us consider \eqref{eq:dissblow} with $u_0, u_1\in L^1_\lloc$.
We assume that~$a\in L^\infty_\lloc([0,\infty)), b\in W^{1,\infty}_\lloc([0,\infty)), f\in L^\infty_\lloc([0,\infty)\times \R^n)$.\\
We say that~$u\in L^1_\lloc([0,\infty)\times\R^n)$ is a global \emph{weak solution} to~\eqref{eq:dissblow} if $|u|^pf \in L^1_\lloc([0,\infty)\times\R^n)$ and for any~$\Phi\in\mathcal{C}_c^\infty([0,+\infty)\times \R^{n})$ it holds
\begin{multline*}
\int_0^\infty\int_{\R^n} u(t,x) \left( \Phi_{tt}(t,x) - a(t) \triangle \Phi(t,x) - b(t) \Phi_t(t,x) - b'(t)\Phi(t,x)\right) \,dxdt \\
    = \int_0^\infty \int_{\R^n} f(t,x)\,|u(t,x)|^p \Phi(t,x)\,dxdt \\
        + \int_{\R^n} \left( (u_1(x)+u_0(x)\,b(0))\,\Phi(0,x) - u_0(x)\,\Phi_t(0,x) \right)\,dx \,.
\end{multline*}
\end{Def}
\noindent
It follows that classical solutions are weak solutions, and that $\mathcal C^2$ weak solutions are classical solutions.

We consider a class of damping coefficients~$b(t)$ satisfying the following.
\begin{Hyp}\label{Hyp:blowup}
Let~$b\in\mathcal{C}^1([0,\infty))$ be such that $b(t)>0$ for any $t\geq0$. We assume that
\begin{align}
\label{eq:liminfblow} \liminf_{t\to\infty} \frac{b'(t)}{b(t)^2} >-1\,,\\
\label{eq:b1blow} \limsup_{t\to\infty} \frac{tb'(t)}{b(t)} <1\,.
\end{align}
\end{Hyp}
\begin{Def}
Let us define
\[ B(t)\doteq \int_0^t \frac1{b(\tau)}\,d\tau. \]
Since~$b(t)>0$, the function~$B:[0,\infty)\to [0,\infty)$ is strictly increasing.
\end{Def}
In  Remark~\ref{Rem:lambda}, using~\eqref{eq:b1blow}, we will show that $b(t)\lesssim t^m$ for some~$m\in[0,1)$, therefore $1/b\not\in L^1$ so that $B(t)$ is
a bjiection.
\begin{Def}\label{Def:beta}
Let~$b(t)>0$ for any $t>0$. We define
\begin{equation}\label{eq:beta}
\beta(t)\doteq \exp\left(-\int_0^t b(\tau)\,d\tau\right).
\end{equation}
\end{Def}
In  Remark~\ref{Rem:betaL1}, using~\eqref{eq:b1blow}, we will show that~$\beta\in L^1(0,\infty)$. Therefore, for any $t>0$, we may define:
\[ \Gamma(t) \doteq\int_t^\infty \beta(\tau)\,d\tau\,,\qquad \hat{b}_1 \doteq (\Gamma(0))^{-1}= \|\beta\|_{L^1(0,\infty)}^{-1}\,.\]
\begin{Thm}\label{Thm:blowup}
Let~$b(t)$ be as in Hypothesis~\ref{Hyp:blowup} and let us assume that
\begin{itemize}
\item $0<a(t)\lesssim B(t)^{-\alpha}$ for some~$\alpha<1$,
\item $f(t,x)\gtrsim B(t)^\gamma\,|x|^\delta$ for some $\gamma>-1$ and~$\delta\in\R$.
\end{itemize}
If
\begin{equation}\label{eq:pcrit}
p_{\min} < p \leq p_\Crit= 1+ \frac{2(1+\gamma)}{n(1-\alpha)} + \frac{\delta}n \,,
\end{equation}
where
\[ p_{\min} = 1 + \max \left\{ \frac{[\gamma+\alpha]^+}{1-\alpha} , \frac{[\delta]^+}n \right\} \,, \]
then there exists no \emph{weak solution} to the Cauchy problem~\eqref{eq:dissblow} with initial data ~$(u_0,u_1)\in L^1$ satisfying
\begin{equation}\label{eq:datablow}
\int_{\R^n} \big(u_1(x)+\hat{b}_1u_0(x)\big) dx > 0\,.
\end{equation}
\end{Thm}
\begin{Rem}
Theorem~\ref{Thm:blowup} is meaningful if $p_{\min} (\alpha,\gamma,\delta)< p_\Crit(\alpha,\gamma,\delta,n)$. Since~$\alpha<1$ and~$\gamma>-1$, this inequality holds if, and only if,
%
\[ \delta > \frac{n[\gamma+\alpha]^+-2(1+\gamma)}{1-\alpha} \,. \]
%
\end{Rem}
In particular, we have the following.
\begin{Example}\label{Ex:special}
If~$\gamma=-\alpha$ in Theorem~\ref{Thm:blowup}, that is, we assume that
\[ 0<a(t)\lesssim (B(t))^{-\alpha} \,, \qquad f(t,x)\gtrsim (B(t))^{-\alpha}\,|x|^\delta\,,\]
for some~$\alpha<1$, and~$\delta>-2$, then a nonexistence result follows for any
\[ 1+\frac{[\delta]^+}n <p\leq1+ \frac{2+\delta}n \,.\]
Moreover, if we replace the assumption on~$f(t,x)$ with~$f(t,x)\gtrsim B(t)^{-\alpha}\,\<x\>^d$, for some~$d>0$, then we may apply Theorem~\ref{Thm:blowup} for any~$\delta\in[0,d]$, obtaining a range for nonexistence given by~$1<p\leq 1+(2+d)/n$.
\\
For~$\delta=0$ we get the range~$1<p\leq 1+2/n$. In particular, we have a counterpart of the global existence results proved in~\cite{DA12+, DLR} for~$p>1+2/n$ when $\alpha=\gamma=\delta=0$.
\end{Example}
\begin{Example}
Let~$\alpha=\delta=0$ in Theorem~\ref{Thm:blowup}, that is, $a(t)$ is bounded and $f(t,x)\gtrsim B(t)^\gamma$ for some~$\gamma>-1$, and $\gamma<2/(n-2)$ if~$n\geq3$. Then a nonexistence result follows for any~$1+[\gamma]^+<p\leq 1+2(1+\gamma)/n$. In particular, we have a counterpart of some global existence results proved in~\cite{DA12}.
\end{Example}
In order to prove Theorem~\ref{Thm:blowup} we will use a \emph{modified} test function method coupled with a careful study of the properties of~$b(t)$.
\\
The test function method is
based on the scaling invariance property of the operator. Since it works for elliptic, parabolic and hyperbolic equations, the corresponding literature is very extensive. We only quote the papers in which Mitidieri and Pohozaev explain how a suitable choice of the test function
gives a nonexistence result. A deep description of this technique can be found in
\cite{MP}, see 
also~\cite{MP1, MP2, MPhyp}.
Here we use a  \it modified \rm test function method: we
apply the scaling argument on an associated equation obtained from the original one by means of a multiplication by an auxiliary function.

\bigskip

The scheme of the paper is the following. In order to clarify our approach in Section~\ref{sec:LiouCP}, we introduce a general notation and we briefly derive a modified version of the test function method for a class of Liouville and Cauchy problems. In Section~\ref{sec:effective} we prove Theorem~\ref{Thm:blowup} as a corollary of the result established in Section~\ref{sec:CP}. In Section~\ref{sec:Examples} we present some other applications of the results of Section~\ref{sec:LiouCP}, in particular for non-damped wave equations and for damped wave equations with a special mass term.

\smallskip

The \it modified \rm test function method can be extended to Liouville problems for differential inequality or for quasilinear operators, as well as one can deal with Cauchy problems for quasilinear systems. For the sake of brevity, we will not investigate here these arguments.

\subsection{Notation}

In this paper, all functions are assumed to be measurable.

\noindent
\begin{itemize}
\item[-] We omit to write $\Rn $ when considering spaces of functions defined on $\Rn$. In particular, $L^p_\lloc$ stands for
$L^p_{\lloc}(\R^n)$ and so on.
\item[-] Given $m:\R^n\to (0,\infty)$, with $L^p_{\lloc}(m(x)dx)$ we denote the space of real functions $f$ such that $mf \in L^p_{\lloc}$
\item[-] With $\mathcal C_c^k(D)$ we denote  the space of functions belonging to
$\mathcal C^k(D)$ with compact support in a domain $D$.
\item[-] Given $p> 1$, by~$p'$ we mean the conjugate of~$p$, that is, $p'(p-1)=p$.
\item[-] By~$x\cdot M$ or $M\cdot x$ we will denote the product of a vector $x\in\R^n$ with a $n\times n$ matrix~$M$ and vice-versa.
\item[-] In what follows, $e_k$ stands for the vector in~$\N^n$ or $\N^{n+1}$ with zero entries, exception given for the $k$-th, which assumes value~$1$.
\item[-] Let $A(\xi),B(\xi)$ be two positive functions on suitable domains.
We write $A\approx B$ if there exist $C_1,C_2>0$ such that
$C_1A(\xi)\le B(\xi)\le C_2A(\xi)$ a.e. Similarly $A(\xi) \lesssim B(\xi)$ means that there exists
$C>0$, independent of $\xi$,  such that $A(\xi)\le CB(\xi) $ a.e.
\end{itemize}

\section{A modified test function method}\label{sec:LiouCP}

In this section we illustrate the \emph{modified} test function method. First we state notation and results for Liouville problems, whose presentation is simpler, then we show how to extend the approach to the Cauchy problems.

\subsection{The Liouville problem}\label{sec:Liou}

We consider the differential Liouville problem:
\begin{equation}\label{eq:Liou}
L(x,\nabla) \, u = f(x)|u|^p\,,\quad x\in \R^N\,,
\end{equation}
where $Lu$ is a linear operator of order~$m$, $p>1$, and $f\in L^\infty_\lloc$ satisfies~$f(x)>0$ a.e. \\
We denote by~$L^*$ the formal adjoint of~$L$: for any~$u,v\in\mathcal{C}_c^\infty$ it holds
\begin{equation}\label{eq.adjoint}
\int_{\R^N} v(x) L(x,\nabla)u(x) \,dx = \int_{\R^N} u(x) L^*(x,\nabla)v(x) \,dx\,.
\end{equation}
Moreover, we denote by $L(x,\xi)$ the symbol of $L$ obtained by means of the Fourier transform.
\\
For Liouville problem~\eqref{eq:Liou} we will use the following definition of solution.
\begin{Def}\label{Def:sol}
Let~$g(x)>0$ a.e. be such that if we put
\[ D(x,\nabla)\doteq g(x)L(x,\nabla)\,, \]
then~$D^*$ is a differential operator with~$L^\infty_\lloc$ coefficients.
\\
We say that~$u\in L_\lloc^1$ is a \it (weak) $g$-solution \rm to~\eqref{eq:Liou} if $|u|^pfg\in L^1_\lloc$ and for any~$\phi\in\mathcal{C}_c^\infty$, it holds
\begin{equation}\label{eq:sol}
\int_{\R^N} f(x)\,g(x)\,|u(x)|^p \phi(x)\,dx = \int_{\R^N} u(x) D^*(x,\nabla)\phi(x)\,dx\,.
\end{equation}
\end{Def}
\begin{Rem}\label{Rem:integrability}
Since we look for $u\in L_\lloc^p(f(x)g(x)dx)$ and $D^*$ has $L^\infty_\lloc$ coefficients,
both the integrals in~\eqref{eq:sol} are well-defined.
\end{Rem}
\begin{Rem}\label{Rem:weak}
If~$g\in\mathcal{C}^\infty$ then Definition~\ref{Def:sol} is equivalent to the definition of \it weak solution: \rm for any~$\Phi\in\mathcal{C}_c^\infty$ it holds
\begin{equation}\label{eq:sold}
\int_{\R^N} f(x)\,|u(x)|^p \Phi(x)\,dx = \int_{\R^N} u(x) L^*(x,\nabla)\Phi(x)\,dx\,.
\end{equation}
Indeed, $D^*\phi=L^*(g\phi)$ and $1/g$ belongs to $\mathcal{C}^\infty$, hence the application $\phi\in\mathcal{C}^\infty \mapsto \Phi=g\phi \in\mathcal{C}^\infty$ is a bijection.
\\
Given $L$ a differential operator of order~$m$ and $u\in L_\lloc^1 \cap L_\lloc^p(f(x)g(x)dx)$ which satisfies \eqref{eq:sol} (respectively \eqref{eq:sold}), a density argument shows that the same relation \eqref{eq:sol} (respectively \eqref{eq:sold})  holds for any $\phi \in \mathcal C^m_c$.
Arguing as before one the equivalence of $g$-solutions and weak solutions for $g\in\mathcal{C}^m$ follows.
\end{Rem}
\begin{Rem}
Let~$L$ be a differential operator of order~$m$ with continuous coefficients and let~$f(x)>0$ be continuous. Let $u(x)$ be a \it classical solution \rm to~\eqref{eq:Liou}, that is, $u\in\mathcal{C}^m$ and $Lu(x)=f(x)|u(x)|^p$. Let~$g(x)>0$ be in~$\mathcal{C}^m$. Then~$D^*$ is a differential operator with continuous coefficients, and $u(x)$ is a $g$-solution. Indeed, $u\in L_\lloc^1 \cap L_\lloc^p(f(x)g(x)dx)$ and
\begin{align*}
\int_{\R^N} f(x)\,g(x)\,|u(x)|^p \phi(x)\,dx& = \int_{\R^N} \phi(x)\,g(x)\, Lu(x) \,dx \\
	&= \int_{\R^N} \phi(x)\,Du(x) \,dx = \int_{\R^N} u(x)\,D^* \,\phi(x)dx \,,
\end{align*}
for any~$\phi\in\mathcal{C}_c^\infty$, where the last inequality holds by density argument. \\ By virtue of Remark~\ref{Rem:weak},  $u$ is a \emph{weak} solution,
as yet known being $u$ a classical solution.
\end{Rem}

\begin{Def}\label{Def:F}
Let~$N\geq2$ and $l>0$. We denote by
\[ C_l \doteq \left\{ x \in \R^N : \ |x_i|\leq l, \ i=1,\ldots,N\right\} \,, \]
the $N$-dimensional cube with length~$2l$, centered at the origin. Moreover, for any~$\alpha\in\N^N$, we put
\[ C^{(\alpha)}_l \doteq \left\{ x\in C_l : |x_i| \geq l/2 \quad \text{for any~$i$ such that~$\alpha_i\neq0$} \right\} \,,\]
and we notice that~$C^{(\alpha)}_l\subset C_l\setminus C_{l/2}$, for any~$|\alpha|\geq 1$.
\\
Now we consider
\[ F_i:(0,\infty)\to(0,\infty)\qquad i=1,\dots,N\,, \]
strictly increasing, continuous functions with~$F_i(R)\to\infty$ as~$R\to\infty$. Let
\[
F(R)\doteq \diag (F_1(R),\ldots,F_N(R))
\]
be the diagonal matrix with~$F_i(R)$ as~$(i,i)$-th entry.
\\
For any~$R>1$, we define the $N$-dimensional rectangle
\[ Q_R \doteq [-F_1(R),F_1(R)]\times \ldots \times [-F_N(R),F_N(R)]\,, \]
that is,
\[ Q_R = \left\{ x\in\R^N : \ \max_i |x_i(F_i(R))^{-1}| \leq 1 \right\}\,\]
is the image of the cube~$C_1$ by the matrix~$F(R)$. We denote the volume of~$Q_R$ by $|Q_k|$; it follows
\[ |Q_R| \doteq 2^N\det F(R) = 2^N \,F_1(R)\cdot\ldots\cdot F_N(R)\,. \]
We also put
\[ Q_R^\sharp \doteq [-F_1(R)/2,F_1(R)/2]\times \ldots \times [-F_N(R)/2,F_N(R)/2]\,, \]
that is, the image of the cube~$C_{1/2}$ by the matrix~$F(R)$. It is clear that~$|Q_R^\sharp|=\det F(R)$.
\\
Finally for any~$\alpha\in\N^N$, we put
\[ Q_R^{(\alpha)} \doteq \left\{ x\in Q_R : |x_i| \geq F_i(R)/2 \quad \text{for any~$i$ such that~$\alpha_i\neq0$} \right\} \,,\]
and we notice that~$Q_R^{(\alpha)}\subset Q_R \setminus Q_R^\sharp$, for any~$|\alpha|\geq 1$.
\end{Def}
\begin{Def}
For any $F(R)$ as in Definition~\ref{Def:F}, we denote by~$S_R$ the scaling operator such that for any $f:\R^N\to \R$ the function~$S_R\,f:\R^N\to\R$ is defined by
\[ S_R f (x) \doteq f \bigl( x\cdot (F(R))^{-1} \bigr) \,. \]
\end{Def}

The test function method is based on a suitable choice of compactly supported function which multiply the considered equation. \\
In one dimensional case we call
\it test function \rm any $\Phi\in \mathcal{C}_c^\infty (\R)$ which satisfies
\begin{itemize}
\item $\Phi(\R)\subset [0,1]$;
\item $\supp \Phi\subset[-1,1]$
\item $\Phi\equiv 1$ on $[-1/2,1/2]$.
\end{itemize}
\begin{Rem}
Clearly, there exists an even $\Phi\in \mathcal{C}_c^\infty (\R)$ satisfying the previous assumptions and decreasing in $[1/2,1]$.
\end{Rem}

In order to consider the $N$-dimensional case, one can use a radial reduction and gain spherical supports or one can separate the variables and use cubic supports. Here we prefer this second procedure.
\begin{Def}\label{Def:psi}
Let $\Phi\in\mathcal \mathcal{C}_c^\infty (\R)$ such that $\supp \Phi\subset[-1,1]$, $\Phi(\R)\subset [0,1]$ and $\Phi\equiv 1$ on $[-1/2,1/2]$.
In what follows by \emph{test function} we mean a function $\psi\in\mathcal{C}_c^\infty(\R^N,[0,1])$ having the following structure:
\begin{equation}\label{eq:psiprod}
\psi(x)= \prod_{j=1}^N \Phi(x_j)\,.
\end{equation}
For any $F(R)$ as in Definition~\ref{Def:F}, we put
\[ \psi_R (x) \doteq S_R \psi (x) \equiv \prod_{j=1}^N \Phi\bigl(x_j\,F_j(R)^{-1}\bigr)\,.\]
\end{Def}
\begin{Rem}
If $\psi$ is a test function, then $\supp \psi \subset C_1$ and $\psi\equiv 1$ in~$C_{1/2}$. Therefore $\supp \psi_R \subset Q_R$ and~$\psi_R\equiv1$ in~$Q_R^\sharp$.
\end{Rem}

\begin{Hyp}\label{Hyp:g}
Let~$g(x)>0$ a.e. be such that~$D^*$ is a differential operator with $L^\infty_\lloc$ coefficients, which contains no zero order terms, that is, $D^*(x,0)=0$. We denote by~$a_\alpha(x)$ its coefficients, i.e.
\[ D^* (x,\nabla) \doteq \sum_{1\leq|\alpha|\leq m} a_\alpha(x) \partial_x^\alpha \,. \]
\end{Hyp}
\begin{Rem}
Given $\psi$ as in Definition \ref{Def:psi}, for any $\alpha \in \N^N$  one has~$\supp \partial_x^\alpha \psi \subset C_1^{(\alpha)}$ and, analogously $\supp \partial_x^\alpha \psi_R \subset Q_R^{(\alpha)}$. In particular Hypothesis~\ref{Hyp:g} implies
\begin{equation}\label{eq:hole}
\supp D^*(\psi^\sigma) \subset C_1\setminus C_{1/2}\,, \qquad \supp D^*(\psi_R) \subset Q_R\setminus Q_R^\sharp\,,
\end{equation}
for any test function $\psi$ and any $\sigma\in \N^*$.
\end{Rem}
\begin{Rem}
It is clear that we also have the following scaling property:
\begin{equation}\label{eq:scaling}
\partial_x^\alpha \psi_R \equiv \partial_x^\alpha (S_R \psi) = \Bigl(\prod_{i=1}^N (F_i(R))^{-\alpha_i} \Bigr) \, S_R (\partial_x^\alpha \psi) \,.
\end{equation}
\end{Rem}

\begin{Lem}\label{lem.power2}
Let $\varphi$ be a test function. For any fixed $r>1$ there exists $\sigma\in \N$ such that
%
\[ |\partial_x^\alpha (\varphi^\sigma)|^r\lesssim \varphi^{\sigma} \,, \quad \text{ for any } \sigma\in \N, \ \sigma \geq |\alpha|\, r'\,.\]
%
\end{Lem}

The proof may be given by inductive argument on $|\alpha| \in \N$.

\smallskip
Now we introduce two quantities related respectively to the linear operator and to the semilinear perturbation.

\begin{Def}\label{Def:HG}
Let~$g(x)$ and $a_\alpha(x)$ be as in Hypothesis \ref{Hyp:g}. For any $\alpha \in \N$ with $1\le |\alpha|\le m$, we put
\begin{align}
\label{eq:HalphaR}
H_\alpha(R)
    & \doteq \prod_{i=1}^N F_i(R)^{-\alpha_i} \,, \\
\label{eq:GalphaR}
G_\alpha(R)
    & \doteq \int_{Q_R^{(\alpha)}} |a_\alpha(x)|^{p'}\,(g(x)\, f(x))^{-(p'-1)}\,dx 
\end{align}
We remark that $G_\alpha(R)$ can be a divergent integral of positive functions, that is, $G_\alpha(R)\in[0,\infty]$.\\
We observe that~$\lim_{R\to\infty} H_\alpha(R)=0$, since~$F_i(R)\to\infty$ for any~$i=1,\ldots,N$.
\end{Def}

Combining the growth behaviour of $H_\alpha(R)$ and $G_\alpha(R)$ we can state a first nonexistence theorem.
\begin{Thm}\label{Thm:main}
Let~$L(x,\nabla)$ be a differential operator, $f(x)>0$ a.e. and $p>1$.
Assume that there exists $g>0$ \it a.e. such that $(L,g)$ satisfies Hypothesis~\ref{Hyp:g}.\\
Suppose that~$u\in L_\lloc^p(f(x)g(x)dx)$ is a global $g$-solution in the sense of Definition~\ref{Def:sol}.\\
If
there exists~$F(R)$ as in Definition~\ref{Def:F}, such that
\begin{equation}\label{eq:GH}
\limsup_{R\to\infty} H_\alpha(R)\,G_\alpha(R)^{\frac1{p'}} <\infty \,\quad \text{ for any } 1\leq|\alpha|\leq m\,,
\end{equation}
then~$u(x)=0$ a.e.
\end{Thm}

Thanks to Remark~\ref{Rem:weak}, from Theorem~\ref{Thm:main} we immediately obtain the following.

\begin{Cor}\label{Cor:weak}
Let~$L(x,\nabla)$ be a differential operator, $f(x)>0$ a.e. and $p>1$.
Let~$u\in L_\lloc^p(f(x)dx)$ be a global weak solution in the sense of~\eqref{eq:sold}.\\
Assume that there exists $g\in \mathcal{C}^m$ such that $(L, g)$ satisfy Hypotheses~\ref{Hyp:g} and
there exists~$F(R)$ as in Definition~\ref{Def:F}, such that \eqref{eq:GH} holds.\\ Then~$u(x)=0$ a.e.
\end{Cor}

\begin{Rem}
Condition \eqref{eq:GH} implies that $H_\alpha(R)\,G_\alpha(R)^{\frac1{p'}}$ are bounded and $G_\alpha(R)$ is finite for sufficiently large $R$.
\end{Rem}

\begin{proof}[Proof of Theorem~\ref{Thm:main}]
Let ($L$, $p$, $f$) and ($g$, $F(R)$) be such that the assumptions of Theorem~\ref{Thm:main} hold. \\
Let~$u\in L_\lloc^p(f(x)g(x)dx)$ be a $g$-solution of $Lu=f|u|^p$, that is, \eqref{eq:sol} is satisfied.
Given $R>1$, and $\psi$ a test function, we put
\begin{align*}
I_R
    & \doteq \int_{\R^N} g(x)f(x)|u(x)|^p\psi_R(x)\,dx \equiv \int_{Q_R} g(x)f(x)|u(x)|^p\psi_R(x)\,dx\,, \\
I_R^\sharp
		& \doteq \int_{Q_R\setminus Q_R^\sharp} g(x)f(x)|u(x)|^p\psi_R(x)\,dx\,.
\end{align*}
By using~\eqref{eq:sol} and~\eqref{eq:hole} together with H\"older inequality, since $p'/p=p'-1$, we formally obtain:
\begin{align}
\nonumber
I_R & = \int_{Q_R\setminus Q_R^\sharp} u D^* \psi_R(x)\,dx \leq \int_{Q_R\setminus Q_R^\sharp} |u|\, |D^* \psi_R(x)| \,dx \\
\label{eq:Holder}
	& \lesssim (I_R^\sharp)^{\frac1p}
\left(\int_{Q_R\setminus Q_R^\sharp}
\frac{|D^* \psi_R(x)|^{p'}}{\psi_R^{p'-1}(x)} \, (g(x)f(x))^{-(p'-1)} \,dx
 \right)^{\frac1{p'}}\,.
\end{align}
By using~\eqref{eq:scaling}, we may now estimate
\[ |D^* \psi_R(x)| \leq \sum_{1\leq|\alpha|\leq m} |a_\alpha(x)|\,|\partial_x^\alpha\psi_R(x)| \leq \sum_{1\leq|\alpha|\leq m} H_\alpha(R)\,|a_\alpha(x)|\,|S_R\partial_x^\alpha\psi(x)| \,. \]
In order to control the quantities
\begin{equation}\label{eq:depsipsi}
\frac{|S_R\partial_x^\alpha\psi(x)|^{p'}}{(\psi_R(x))^{p'-1}} = S_R \left( \frac{|\partial_x^\alpha\psi|^{p'}}{\psi^{p'-1}}\right)(x)\,,
\end{equation}
we choose a particular $\psi$. More precisely we take $\psi=\varphi^\sigma$ a suitable integer power of a test function $\varphi$ with $\sigma \in \N$ and $\sigma\ge mp'$. By virtue of Lemma~\ref{lem.power2}, the quantities in~\eqref{eq:depsipsi} are bounded. In particular, the quantity~$|D^* \psi_R(x)|^{p'}\psi_R^{-(p'-1)}(x)$ in~\eqref{eq:Holder} is well-defined.
\\
Recalling that $\supp \partial_x^\alpha\psi_R(x) \subset Q_R^{(\alpha)}$, we get
\begin{equation}\label{eq.IR}\begin{split}
I_R & \lesssim (I_R^\sharp)^{\frac1p} \sum_{1\leq|\alpha|\leq m} H_\alpha (R)\,\left( \int_{Q_R^{(\alpha)}} |a_\alpha(x)|^{p'}\,(g(x)\, f(x))^{-(p'-1)}\,dx \right)^{\frac1{p'}} \\
    & \approx (I_R^\sharp)^{\frac1p} \sum_{1\leq|\alpha|\leq m} H_\alpha(R) \,(G_\alpha(R))^{\frac1{p'}} \leq C\, (I_R^\sharp)^{\frac1p} \,,
\end{split}
\end{equation}
where~$C>0$ does not depend on~$R$, since~$H_\alpha(R) \,(G_\alpha(R))^{\frac1{p'}}$ are bounded.

We immediately get that~$I_R\leq C_p\doteq C^{\frac{p}{p-1}}$, being $I_R^{1-\frac1p}\leq C$. Indeed, $I_R^\sharp\leq I_R$ and~$p>1$. \\
Now we apply Beppo-Levi convergence theorem. Let~$\{R_k\}_{k\in\N}$ be a strictly increasing sequence with~$R_k\to\infty$ as~$k\to\infty$ and so that~$F_i(R_{k+1})\ge 2 F_i(R_k)$ for any~$i=1,\ldots,N$. Then it holds~$\psi_{R_{k+1}}(x)\geq\psi_{R_k}(x)$ for any~$x\in\R^N$ and $\psi_{R_k}(x)\to 1$ pointwise as~$k\to\infty$. 
We constructed an increasing sequence of positive functions $\{ g\,f\,|u|^p\,\psi_{R_k} \}_{k\in\N}$ to which we can apply Beppo-Levi convergence theorem:
\begin{equation}\label{eq:BL}
I \doteq \int_{\R^N} g(x)f(x)|u(x)|^p\,dx  = \lim_{k\to\infty} I_{R_k} \,.
\end{equation}
This gives~$I\leq C_p$, hence~$u\in L^p(f(x)g(x)dx)$. Now we can apply Lebesgue convergence theorem to gain
\[ \lim_{R\to\infty} I_R^\sharp \leq \lim_{R\to\infty} \int_{Q_R\setminus Q_R^\sharp} g(x)f(x)|u(x)|^p\,dx =0 \,. \]
This concludes the proof, since~$I_R\leq C\, (I_R^\sharp)^{\frac1p} \to0$ as~$R\to\infty$ so that $I=0$ thanks to~\eqref{eq:BL}, hence $u=0$ a.e. since~$g(x)f(x)>0$ a.e.
\end{proof}

\begin{Rem}
Theorem~\ref{Thm:main} cannot be extended  if $D^*$ has zero order term, namely $a_0(x)\neq0$.
Indeed with the same notation of Theorem~\ref{Thm:main}, we would have $Q_R^{(0)}=Q_R$, so that $I_R^\sharp=I_R$ and \eqref{eq.IR} gives
$I_R\lesssim I_R^{\frac1p}$. The final convergence argument does not work except for the case $H_0(R)G_0(R)\to 0$ for $R\to \infty$.
On the other hand,
$H_0(R)=1$,  and
\[ \lim_{R\to\infty} G_0(R) \equiv \lim_{R\to\infty} \int_{Q_R} |a_0(x)|^{p'}\,(g(x)\, f(x))^{-(p'-1)}\,dx  = \int_{\R^N} |a_0(x)|^{p'}\,(g(x)\, f(x))^{-(p'-1)}\,dx\,. \]
So that  $H_0(R)G_0(R)\to 0$ if and only if $a_0(x)=0$ a.e.\\
This means that $g$ has to be taken, when it exists, such that the corresponding $D^*$ does not contain the zero order term.
\end{Rem}

%

\subsection{The Cauchy Problem}\label{sec:CP}

Let us denote~$\R_+^{n+1}\doteq [0,\infty)\times\R^n$. Here we consider the Cauchy problem:
\begin{equation}\label{eq:CP} \begin{cases}
Lu = f(t,x)|u|^p\,, \qquad t\geq0, \ x\in\R^n \\
\partial_t^ju(0,x)=u_j(x)\,, \quad j=0,\ldots,m-1
\end{cases}
\end{equation}
where $u_j\in L^1_{\lloc}$ for any $j=0,\dots,m-1$ and
\begin{equation}\label{eq:LCauchy}
L = \partial_t^m +\sum_{j=0}^{m-1} L_j(t,x,\nabla_x) \partial_t^j \,,
\end{equation}
with $L_j$ differential operators of order~$m_j\in\N$. We also put~ $L_m=1$.
\\
Let~$g(t,x)>0$ a.e. and let us define $D=gL$ and $D_j=gL_j$ for $j=0,\dots,m-1$. Formally
\begin{equation}\label{eq:DCauchy}
D^* v (t,x) \doteq \sum_{j=0}^m  (-1)^j \partial_t^j \left( D_j^*(t,x,\nabla_x)\,v (t,x) \right) \,,  \quad \text{for any} \quad v\in \mathcal{C}^\infty ([0,\infty)\times\R^n)\,.
\end{equation}
Here $D_m=D_m^*=g$, and $D_j^*=(gL_j)^*$ according to notation \eqref{eq.adjoint}.
\\
For any~$j=0,\ldots,m$, we assume that $D_j^*(t,x,\nabla)$ is a differential operator with~$\mathcal{C}^j([0,\infty),L^\infty_\lloc)$ coefficients.
In particular we suppose $g\in\mathcal{C}^m([0,\infty),L^\infty_\lloc)$.
\\
Let~$\eta\in\mathcal{C}_c^\infty([0,\infty))$ be such that~$\eta\equiv1$ in some neighborhood of~$\{t=0\}$. Integrating by parts, we obtain
\begin{align*}
\int_0^\infty \int_{\R^n} \eta(t)\phi(x)\,D\,u(t,x)\,dx\,dt
    & = \int_0^\infty \int_{\R^n} u(t,x)\,D^*\,\bigl(\eta(t)\phi(x)\bigr)\,dx\,dt \\
    & \quad - \sum_{j=1}^m \sum_{k=1}^j (-1)^{j-k} \int_{\R^n}[\partial_t^{k-1} u(0,x)] \, (\partial_t^{j-k} D_j^*)(0,x,\nabla_x) \phi(x)\,dx\,,
\end{align*}
for any~$u\in \mathcal{C}^\infty([0,\infty)\times \R^n)$ and~$\phi\in\mathcal{C}_c^\infty(\R^n)$.
\begin{Def}\label{Def:solCP}
We say that~$u\in L^1_\lloc$ is a $g$-solution to~\eqref{eq:CP} if $|u|^pfg\in L^1_\lloc$ and for any~$\eta\in\mathcal{C}_c^\infty([0,\infty))$ such that~$\eta\equiv1$ in some neighborhood of~$\{t=0\}$ and for any~$\phi\in\mathcal{C}_c^\infty(\R^n)$ it holds
\[
\int_0^\infty \int_{\R^n} g(t,x)\,f(t,x)\,|u(t,x)|^p\,\eta(t)\phi(x)\,dx\,dt =
\int_0^\infty \int_{\R^n} u(t,x)\,D^*\,\bigl(\eta(t)\phi(x)\bigr)\,dx\,dt - K_0(\phi)\,,
\]
where we put
\begin{equation}\label{eq:Ku}
K_0(\phi) \doteq \sum_{j=1}^m \sum_{k=1}^j (-1)^{j-k} \int_{\R^n} u_{k-1}(x) \,( \partial_t^{j-k} D_j^*)(0,x,\nabla_x) \phi(x)\,dx \,,
\end{equation}
with~$u_{k-1}(x)$ as in~\eqref{eq:CP}.
\end{Def}
We remark that we are not assuming that $L$ is a Kovalevskian operator in  normal form. Indeed
$L$ is a differential operator of order $M:=\max \{ m, j+m_j\,|\, j=0,\dots, m-1\}$, possibly $M>m$.
In order to state our result, analogously to Hypothesis~\ref{Hyp:g}, we assume that $D^*$ may be written as
\begin{equation}\label{eq:DCP}
D^* (t,x,\nabla) \doteq \sum_{1\leq\alpha_0+|\alpha|\leq M} a_{\tilde \alpha}(t,x) \partial_t^{\alpha_0}\partial_x^\alpha\,.
\end{equation}
Let us extend the notation of Section~\ref{sec:Liou} for~$N=n+1$, denoting~$x_0=t$ when it is convenient and $\widetilde{x}=(x_0,x)\in\R_+^{n+1}$ with~$x_0\geq0$ and~$x\in\R^n$. We put
\[ \widetilde{C}_l \doteq \left\{ \widetilde{x} \in \R_+^{n+1} :\ x_0\leq l, \ |x_i|\leq l \right\} \equiv [0,l] \times C_l \,, \]
and similarly
\begin{gather*}
\widetilde{F}(R)\doteq \bigl(F_0(R),F(R)\bigr) \equiv \bigl(F_0(R),F_1(R),\ldots,F_n(R)\bigr)\,, \\
\widetilde{Q}_R \doteq [0,F_0(R)] \times Q_R\,, \qquad \widetilde{Q}_R^\sharp \doteq [0,F_0(R)/2] \times Q_R^\sharp \,.
\end{gather*}
We remark that~$\det \widetilde{F}(R)= F_0(R)\,\det F(R)$. We will consider test functions with separable variables, that is,
\begin{equation}\label{eq:wtpsi}
\widetilde{\psi}(\widetilde{x}) \equiv \widetilde{\psi}(t,x)=\eta(t)\psi(x) \equiv \eta(t)\,\prod_{j=1}^n\phi(x_j) \,,
\end{equation}
with $\phi$ as in Definition \ref{Def:psi} and $\eta$ satisfying
\begin{itemize}
\item $\eta\in\mathcal C_c^\infty ([0,+\infty))$;
\item $\eta([0,+\infty))\subset [0,1]$;
\item $\supp \eta\subset[0,1]$
\item $\eta\equiv 1$ on $[0,1/2]$.
\end{itemize}
This choice is consistent with our definition of~$\widetilde{C}_l$, in particular $\supp \widetilde{\psi} \subset \widetilde{C}_1$ and $\widetilde{\psi}(t,x)=1$ for any~$(t,x)\in \widetilde{C}_{1/2}$.\\
For any~$\widetilde{\alpha}=(\alpha_0,\alpha)\in\N^{n+1}$, 
we define
\begin{align*}
\widetilde{Q}_R^{(\widetilde{\alpha})}
    & =\widetilde{Q}_R^{(\alpha_0,\alpha)}
    = \left\{ (x_0,x)\in \widetilde{Q}_R : \ |x_i|\geq F_i(R)/2\,, \ \text{for any~$i=0,\dots,m$ such that $\alpha_i\neq0$} \right\}\,, \\
H_{\widetilde{\alpha}}(R)
    &=H_{(\alpha_0,\alpha)}(R)
    = \prod_{i=0}^n (F_i(R))^{-\alpha_i}\,,
\intertext{and for any~$a_{\widetilde{\alpha}}$ as in~\eqref{eq:DCP} we put}
G_{\widetilde{\alpha}}(R)
    &=G_{(\alpha_0,\alpha)}(R)
     = \int_{\widetilde{Q}_R^{(\alpha_0,\alpha)}} \frac{|a_{(\alpha_0,\alpha)}(t,x)|^{p'}}{(g(t,x)\, f(t,x))^{p'-1}}\,dt\,dx \,.
\end{align*}

\begin{Thm}\label{Thm:mainCP}
Let~$L(t,x,\partial_t,\nabla_x)$ be a differential operator, $f(t,x)>0$ a.e. and $p>1$. Assume that there exists~$g(t,x)>0$ a.e. such that
$D^*$ does not contain a term of zero order. Let us assume that there exists~$\widetilde{F}(R)$ such that
\begin{equation}\label{eq:GHCP}
\limsup_{R\to\infty} H_{\tilde \alpha}(R)\,G_{\tilde\alpha}(R)^{\frac1{p'}} <\infty \,\quad \text{ for any } 1\leq|\tilde \alpha|\leq M\,.
\end{equation}
Let $(u_0,u_1,\dots, u_{m-1})$ be sufficiently smooth so that $U_0\in L^1$ where
\begin{equation}\label{eq:data}
U_0(x) \doteq \sum_{j=1}^m \sum_{k=1}^j (-1)^{j-k}(\partial_t^{j-k} D_j)(0,x,\nabla_x) u_{k-1}(x)\,.
\end{equation}
If
\begin{equation}\label{eq:DATAcond}
\int_{\R^n} U_0(x)\,dx >0\,
\end{equation}
then there exists no global $g$-solution~$u(t,x)$ to~\eqref{eq:CP} in the sense of Definition~\ref{Def:solCP}.
\end{Thm}
\begin{Rem}
Analogously to Corollary~\ref{Cor:weak}, if~$g$ is sufficiently smooth then the non-existence result of $g$-solution in Theorem \ref{Thm:mainCP} implies a non-existence result of global weak solution.
\end{Rem}
\begin{Rem}
Here and in the examples we will not discuss the local existence result for the equation $Lu=f(t,x)\,|u|^p$. Theorem~\ref{Thm:mainCP}  as well as its applications means that either the local solution blows up  or the local solution does not exist. No information is contained on the blow up mechanism.
\end{Rem}
\begin{Rem}\label{Rem:datareg}
If $u_k\in W^{M_k,1}$ where $M_k=\max\{m_j\,|k\le j\le m-1\}$, then $U_0\in L^1$.
\end{Rem}
\begin{Rem}\label{Rem:LucPez}
The expression of $U_0$ simplifies if $D_j$ are independent of $t$. For example this holds if~$g=g(x)$ and $L$ has constant coefficient with respect to $t$.
\\
In particular if $L_1=\dots =L_h\equiv0$ for some $h=1,\dots, m-1$ and $D_j$ is independent of $t$ for $j=h+1,\dots,m-1$ then condition \eqref{eq:DATAcond} only involves the initial data $u_h,\ldots,u_{m-1}$, and no sign assumption is needed for $u_0, \ldots, u_{h-1}$. More in general, the same property holds if $L_1=\dots =L_h\equiv0$ and~$\partial_t^{j-k}D_j(0,x,\nabla_x)=0$ for~$j=h+1,\ldots,m-1$ and $k=0,\ldots,h$.
\end{Rem}
\begin{proof}
We follow the proof of Theorem~\ref{Thm:main} with minor modifications. \\
Let~$u(t,x)$ be a $g$-solution to~\eqref{eq:CP}, in the sense of Definition~\ref{Def:solCP}.
Due to \eqref{eq:DATAcond} it follows that some data are not zero, so that $u\equiv0$ is not a solution to~\eqref{eq:CP}. We put
\[ I=\int_0^\infty \int_{\R^n} g(t,x)\,f(t,x)\,|u(t,x)|^p dxdt \in [0,\infty]\,.\]
and we shall prove that~$I=0$. This gives absurd conclusion and the theorem follows. \\
Being $U_0\in L^1$, it holds
\begin{equation}\label{eq:K0}
K_0(\phi) = \int_{\R^n} U_0(x) \, \phi(x) \, dx\,,
\end{equation}
for any~$\phi\in\mathcal{C}_c^\infty(\R^n)$. In particular,
there exists~$\overline{R}>0$ such that~$K_0(\psi_R)\geq0$ for any~$R\geq\overline{R}$.
Therefore, if we put
\[ I_R \doteq \int_0^\infty \int_{\R^n} g(t,x)\,f(t,x)\,|u(t,x)|^p \psi_R(x) \eta_R(t) dxdt \,, \]
then we obtain
\begin{equation}\label{eq:dopodato}
I_R \leq \int_0^\infty \int_{\R^n} u(t,x)\,D^*\,\bigl(\eta_R(t)\psi_R(x)\bigr)\,dx\,dt \,,
\end{equation}
for~$R\geq\overline{R}$. We can now follow the proof of Theorem~\ref{Thm:main}, obtaining $I_R\lesssim (I_R^\sharp)^{\frac1p}$ where
\[
I_R^\sharp := \int_{F_0(R)/2}^{F_0(R)} \int_{Q_R\setminus Q_R^\sharp} g(t,x)\,f(t,x)\,|u(t,x)|^p\eta_R(t)\psi_R(x)\,dx\,dt \,.
\]
This implies $I_R$ bounded with respect to $R$ and hence $I$ finite since $I_{R_k}\to I$ for a suitable sequence~$R_k\to\infty$. One can apply Lebesgue convergence theorem and get $I_{R}^\sharp\to 0$. In turn, this gives $I_R\to 0$ as~$R\to\infty$, hence $I=0$.
\end{proof}
\begin{Rem}\label{Rem:pointpos}
Let $(u_0,u_1,\dots, u_{m-1})$ be sufficiently smooth so that $U_0$ is defined, not necessarily in $L^1$, and~$U_0(x)\geq0$ a.e. Then inequality~\eqref{eq:dopodato} still holds, as well as the nonexistence result in Theorem~\ref{Thm:mainCP}.
\end{Rem}


\section
{proof of Theorem~\ref{Thm:blowup}}\label{sec:effective}

In order to prove Theorem~\ref{Thm:blowup} we first prepare some instruments.
\begin{Rem}
As a consequence of~\eqref{eq:liminfblow} we obtain
\begin{equation}\label{eq:liminftb}
\liminf_{t \to \infty} tb(t)>1\,.
\end{equation}
Indeed, there exist a sufficiently small $\delta\in(0,1)$ and $T>0$ such that~$b'(t)/b(t)^2>-(1-\delta)$ for any~$t\geq T$. It is now sufficient to integrate in~$[T,t]$, obtaining
\[
\frac1{b(T)}-\frac1{b(t)} > -(1-\delta)(t-T)
\]
hence
\[
b(t) > \frac1{1/b(T)+(1-\delta)(t-T)}\,.
\]
\end{Rem}
\begin{Rem}
Combining~\eqref{eq:liminftb} with~\eqref{eq:b1blow}, we may prove that~$b'(t)/b(t)^2<1$ for large~$t$. Since $b'(t)/b(t)^2$ is continuous, recalling~\eqref{eq:liminfblow}, we may conclude that
\begin{equation}
\label{eq:beffblow}
|b'(t)|\leq Cb^2(t)\,.
\end{equation}
\end{Rem}
\begin{Rem}\label{Rem:lambda}
On the other hand, combining~\eqref{eq:liminftb} with~\eqref{eq:beffblow} and recalling \eqref{eq:b1blow}, we derive
\begin{equation}\label{eq:Mmdiff}
\frac{-M}t \leq \frac{b'(t)}{b(t)} \leq \frac{m}t\,,
\end{equation}
for some~$M\geq0$ and~$m\in[0,1)$, for~$t\geq T$ with a sufficiently large~$T>0$. Integrating~\eqref{eq:Mmdiff} in~$[t, \lambda t]$
 and taking the exponential, it follows
\begin{equation}\label{eq:blambdat}
\lambda^{-M} \leq \frac{b(\lambda t)}{b(t)} \leq \lambda^m\,,
\end{equation}
for any $t\ge T$ and $\lambda \ge 1$.
\\
Integrating~\eqref{eq:Mmdiff} in~$[T,t]$ we get~$b(t)\leq C \,t^m$ for any~$t\geq T$, where~$C=C(T,m)$.
\end{Rem}
\begin{Rem}\label{Rem:betaL1}
By using~\eqref{eq:liminftb} we can prove that $\beta\in L^1(0,\infty)$ and that
\begin{equation}\label{eq:limbbeta}
\lim_{t\to\infty} \frac{\beta(t)}{b(t)} =0\,,
\end{equation}
where~$\beta(t)$ is defined in~\eqref{eq:beta}. Indeed, there exists $\delta>0$ and~$T>0$ such that
\[ \exp \left(-\int_T^t b(\tau)\,d\tau\right) \lesssim \exp \left(-(1+\delta)\int_T^t \frac1\tau\,d\tau\right) \leq t^{-(1+\delta)}\,, \]
for any~$t\geq T$. Therefore~$\beta\in L^1$. Moreover, $b(t)/\beta(t) \leq C t^{m-(1+\delta)}\to0$ as~$t\to\infty$, due to Remark~\ref{Rem:lambda}.
\end{Rem}
\begin{Rem}\label{Rem:b2bprime}
Thanks to~\eqref{eq:beffblow} and~\eqref{eq:liminfblow}, there exist~$\epsilon,C>0, T>0$ such that
\begin{equation}\label{eq:b2bprime}
\epsilon \leq \frac{b(t)^2+b'(t)}{b(t)^2} < C  \quad  \text{for any $t\geq T$.}
\end{equation}
In particular, $b(t)^2+b'(t)>0$, hence $\beta(t)/b(t)$ is strictly decreasing for~$t\geq T$.
\end{Rem}
\begin{Rem}\label{Rem:approxG}
We claim that:
\begin{equation}\label{eq:Gammab}
\Gamma(t)\approx \frac{\beta(t)}{b(t)} = \frac1{b(t)} \exp\left(-\int_0^t b(\tau)\,d\tau\right)\,, \qquad \text{for any~$t\in[0,\infty)$.}
\end{equation}
In order to prove \eqref{eq:Gammab}, we may integrate the relation
\[ \left(-\frac{\beta(t)}{b(t)}\right)'= \beta(t) \left( \frac{b'(t)}{b^2(t)} +1 \right) \approx \beta(t)\qquad t\ge T \,, \]
where the equivalence is a consequence of~\eqref{eq:b2bprime}, and $T>0$ is the same found in Remark~\ref{Rem:b2bprime}. Due to \eqref{eq:limbbeta}, this gives us $\Gamma(t)\approx \beta(t)/b(t)$ in~$[T,\infty)$. On the other hand, $\Gamma(t)\,b(t)/\beta(t)$ is a strictly positive continuous function in~$[0,T]$, and this concludes the proof of~\eqref{eq:Gammab}.
\end{Rem}
\begin{Rem}\label{Rem:approxB}
We claim that
\begin{equation}\label{eq:Btb}
B(t)\approx t/b(t)\,, \qquad \text{for any~$t\in[0,\infty)$.}
\end{equation}
Similarly to Remark~\ref{Rem:approxG}, using~\eqref{eq:b1blow} and \eqref{eq:Mmdiff} one can prove that
\begin{equation}\label{eq:bbB}
\left(\frac{t}{b(t)}\right)'= \frac1{b(t)} \left( 1 - \frac{tb'(t)}{b(t)} \right) \approx \frac1{b(t)}\,,
\end{equation}
for any $t\geq \tilde T$, for a suitable $\tilde T>0$. Indeed, from \eqref{eq:Mmdiff} one can deduce $-{\tilde M}b(t)\le tb'(t)$ for any $t\geq0$ and for a suitable $\tilde M >0$. It follows that $t/b(t)$ is strictly increasing for~$t\geq \tilde T$, and that $B(t)\approx t/b(t)$ in~$[\tilde T,\infty)$.
On the other hand, the function~$t|b'(t)|/b(t)$ is continuous in~$[0,\tilde T]$, therefore there exists~$t_0\in(0,\tilde T]$ such that~$t|b'(t)|/b(t)<1$ for any~$t\in[0,t_0]$, hence~\eqref{eq:bbB} remains true in $[0,t_0]$ and $B(t)\approx t/b(t)$ in this interval. Finally, if $t_0<\tilde T$, being $B(t)$ and $t/b(t)$ continuous strictly positive functions in~$[t_0,\tilde T]$, we can conclude that~$B(t)\approx t/b(t)$ in~$[t_0,\tilde T]$ too. This concludes the proof of~\eqref{eq:Btb}.
\\
Recalling \eqref{eq:blambdat} we get
\begin{equation}\label{eq:Blambdan}
\lambda^{1-m} \lesssim \frac{B(\lambda t)}{B(t)} \lesssim \lambda^{1+M}\,,
\end{equation}
for any $t\geq T$ and $\lambda \ge 1$.
\end{Rem}


We are now ready to prove Theorem~\ref{Thm:blowup}.
\begin{proof}
We define~$g(t)\doteq \Gamma(t)/\beta(t)$; this is the solution to the following Cauchy problem:
\begin{equation}\label{eq:geq}
\begin{cases}
-g'(t)+g(t)b(t)=1, \quad t>0, \\
g(0)=\frac1{\hat{b}_1}.
\end{cases}
\end{equation}
We remark that~$g\in\mathcal{C}^2$, since~$b\in\mathcal{C}^1$, therefore Remark~\ref{Rem:weak} is applicable and we will establish non-existence results for the weak solution to~\eqref{eq:dissblow}.
\\
With this choice of $g(t)$, putting $D=g(t)(\partial_{tt}-a(t)\Delta+b(t)\partial_t)$, we see that $D^*$ does not contain the zero order term:
\[
D^* = g(t)\partial_t^2 - g(t)\,a(t)\triangle + (g'(t)-1)\partial_t\,.
\]
Moreover
\begin{align*}
U_0(x)= \left( g(t)b(t) u_0(x) - g'(t) u_0(x) + g(t) u_1(x) \right)_{t=0} = u_0(x) + \frac1{\hat{b}_1} u_1(x)\,.
\end{align*}
and the initial data condition~\eqref{eq:datablow} is equivalent to~\eqref{eq:DATAcond}.
\\
Let us define~$A:[0,\infty)\to[0,\infty)$ to be the inverse of the function~$B(t)$. It follows that~$A$ is strictly increasing and bijective. We set~$F_0(R)\doteq A(R^d)$ with a suitable~$d>0$ which we will choose later, and~$F_i(R)=R$ for any~$i=1,\ldots,n$, so that
\[ H_{2e_0}(R) = A(R^d)^{-2}\,, \qquad H_{e_0}(R) = A(R^d)^{-1} \,, \qquad H_{2e_i}(R) = R^{-2}\,, \qquad i=1,\ldots,n\,. \]
By using the growth assumptions on $f(t,x)$ and $a(t)$, we obtain
\begin{align*}
G_{2e_0}(R)
	& \lesssim \int_{A(R^d)/2}^{A(R^d)} g(t)\, B(t)^{-\gamma(p'-1)} \int_{Q_R} |x|^{-\delta(p'-1)} dx\,dt \,,\\
G_{e_0}(R)
	& \lesssim \int_{A(R^d)/2}^{A(R^d)} \frac{|1-g'(t)|^{p'}}{(g(t)\,B(t)^\gamma)^{p'-1}} \int_{Q_R} |x|^{-\delta(p'-1)} dx\,dt \,,\\
G_{2e_i}(R)
	& \lesssim  \int_0^{A(R^d)} g(t)\, B(t)^{-\gamma(p'-1)-\alpha p'} \int_{Q_R\setminus Q_{R/2}} |x|^{-\delta(p'-1)} dx\,dt\,.
\end{align*}
It is clear that
\[ \int_{Q_R} |x|^{-\delta(p'-1)} dx \approx R^{n-\delta(p'-1)} \,,\]
since~$\delta(p'-1)<n$, thanks to the assumption
\begin{equation}\label{eq:pmin1}
p>1+\frac\delta{n}\,.
\end{equation}
By using~\eqref{eq:Gammab} we have $g(t)=\Gamma(t)/\beta(t) \approx 1/b(t)$; recalling the definition of~$B(t)$ and~$A(s)$ it follows
\begin{align*}
\int_0^{A(R^d)} g(t)\,B(t)^{-\alpha p'}\, B(t)^{-\gamma(p'-1)} dt
	& \lesssim \int_0^{A(R^d)} \frac1{b(t)}\, B(t)^{-\alpha p'-\gamma(p'-1)}\, dt \\
	& \approx B(t)^{1-\alpha p'-\gamma(p'-1)}|_{t=A(R^d)} = R^{d(1-\alpha p'-\gamma(p'-1))}\,.
\end{align*}
Indeed, $\alpha p'+\gamma(p'-1)<1$ thanks to the assumption
\begin{equation}\label{eq:pmin2}
(1-\alpha) \, p> \gamma+1 \,.
\end{equation}
We remark the assumption $p>p_{\min}$ in Theorem~\ref{Thm:blowup} implies \eqref{eq:pmin1} and \eqref{eq:pmin2}.

Therefore we have
\[ G_{2e_i}(R)\lesssim R^{n+d-d\alpha p'-(d\gamma+\delta)(p'-1)}\,, \qquad i=1,\ldots,n\,.\]
To deal with $G_{2e_0}$ we compute
\[
\int_{A(R^d)/2}^{A(R^d)} g(t)\,B(t)^{-\gamma(p'-1)} dt \lesssim
\begin{cases}
B(t)^{1-\gamma(p'-1)}|_{t=A(R^d)} & \text{ if } \gamma(p'-1)<1,\\
\log B(t)|_{t=A(R^d)}-\log B(t)|_{t=A(R^d)/2} & \text{ if }\gamma(p'-1)=1,\\
B(t)^{1-\gamma(p'-1)}|_{t=A(R^d)/2} & \text{ if }\gamma(p'-1)>1.
\end{cases}
\]
Recalling that~$A:[0,\infty)\to[0,\infty)$ is a strictly increasing, surjective function, and that~$d>0$, it follows that $A(R^d)\to\infty$ as $R\to \infty$. Therefore we get $B(A(R^d))\approx B(A(R^d)/2)$ for large~$R$, thanks to~\eqref{eq:Blambdan}. Assuming \eqref{eq:pmin1}, in all the three cases we conclude that
\[ G_{2e_0} (R)\lesssim R^{n+d-(d\gamma+\delta)(p'-1)}\, ,\]
for large~$R$.
It remains to estimate $G_{e_0}(R)$. By using $1-g'=2-gb$ and $g\approx 1/b$, see~\eqref{eq:Gammab}, it follows that $1-g'$ is bounded. Hence
\[
\int_{A(R^d)/2}^{A(R^d)} \frac{|1-g'(t)|^{p'}}{g(t)^{p'-1}}\, B(t)^{-\gamma(p'-1)} dt
	 \lesssim \int_{A(R^d)/2}^{A(R^d)} b(t)^{p'}\,\frac1{b(t)}\, B(t)^{-\gamma(p'-1)} dt\,,
\]
for a sufficiently large~$R$. Now, we use~\eqref{eq:blambdat}. In particular, we have $b(s)\approx b(t)$ for any~$s\in[t/2,t]$, for sufficiently large~$t$, namely~$t\geq 2T$. Proceeding as before, we find
\[
G_{e_0}(R)
	\lesssim b(A(R^d))^{p'} R^{n+d-(d\gamma+\delta)(p'-1))} \,,
\]
for large~$R$. On the other hand, the equivalence \eqref{eq:Btb} implies $b(A(R^d))\approx A(R^d)\,R^{-d}$, so that we can conclude
\[
G_{e_0}(R)
	\lesssim (A(R^d))^{p'} R^{n-(d+d\gamma+\delta)(p'-1)}\,,
\]
for large~$R$. Summarizing,
\begin{align}
\label{eq:H2i}
H_{2e_i}(R)\,(G_{2e_i}(R))^{\frac1{p'}}
	& \lesssim R^{-2-d\alpha-d\gamma-\delta+\frac{n+\delta+d(1+\gamma)}{p'}}\,, \qquad i=1,\ldots,n\,, \\
\label{eq:H20}
H_{2e_0}(R)\,(G_{2e_0}(R))^{\frac1{p'}}
	& \lesssim A(R^d)^{-2}\,R^{-d\gamma-\delta+\frac{n+\delta+d(1+\gamma)}{p'}}\,,\\
\label{eq:H10}
H_{e_0}(R)\,(G_{e_0}(R))^{\frac1{p'}}
	& \lesssim R^{-d-d\gamma-\delta+\frac{n+\delta+d(1+\gamma)}{p'}}\,,
\end{align}
for large~$R$. But the estimate in~\eqref{eq:H20} is better than the estimate in~\eqref{eq:H10}, since $R^d\lesssim A(R^d)^2$. Indeed, this is equivalent to~$B(t)\lesssim t^2$, which holds true due to~\eqref{eq:Btb} and~\eqref{eq:liminftb}. 
Looking for~$d>0$ such that the estimates in~\eqref{eq:H2i} and~\eqref{eq:H10} are equal, we immediately find~$d=2/(1-\alpha)$.
\\
Therefore condition~\eqref{eq:GHCP} holds if
\[ d(1+\gamma) + \delta \geq \frac{n+d(1+\gamma)+\delta}{p'} \,, \qquad \text{i.e.} \qquad p\leq 1 + \frac{d(1+\gamma)+\delta}n = 1 + \frac{2(1+\gamma)}{n(1-\alpha)}+\frac\delta{n} \,. \]
By applying Theorem~\ref{Thm:mainCP}, we conclude the proof.
\end{proof}

\subsection{Examples for the damping term}

\begin{Example}\label{Ex:bk}
Let us choose
\begin{equation}\label{eq:bk}
b(t)=\frac\mu{(1+t)^\expo} \quad \text{for some~$\mu>0$ and~$\expo\in(-1,1]$.}
\end{equation}
Being $\expo\in(-1,1]$, Hypothesis~\ref{Hyp:blowup} holds, provided that~$\mu>1$ if~$\kappa=1$. Indeed 
\begin{align*}
\lim_{t\to\infty} \frac{b'(t)}{(b(t))^2}
    & = -\frac\expo\mu\,\lim_{t\to\infty} \frac1{(1+t)^{1-\expo}} = \begin{cases}
-1/\mu & \text{if~$\expo=1$,}\\
0 & \text{if~$\expo\in(-1,1)$,}
\end{cases} \\
\lim_{t\to\infty} \frac{t\,b'(t)}{b(t)}
    & = -\expo \lim_{t\to\infty} \frac{t}{1+t} = -\expo \,.
\end{align*}
We notice that $B(t)\approx t^{1+\expo}$.
\end{Example}
\begin{Example}\label{Ex:bkplus}
We may consider perturbations of~\eqref{eq:bk} by taking
\[ b(t)=\frac\mu{(1+t)^\expo}\,v(t)\,, \qquad \text{with~$\mu>0$ and~$0<|\kappa|<1$,} \]
where $v\in \mathcal{C}^1([0,\infty))$ satisfies $v(t)>0$ and
\begin{equation}\label{eq:vpert}
\lim_{t\to\infty} \frac{tv'(t)}{v(t)} = 0\,.
\end{equation}
It immediately follows that
\begin{align*}
\frac{b'(t)}{b(t)}
    & = -\frac\expo{1+t} \left( 1 + \frac{(1+t)\,v'(t)}{\expo\,v(t)}\right)\,,
\intertext{has the same asymptotic profile of~$-\expo/(1+t)$, due to~\eqref{eq:vpert}. Therefore}
\liminf_{t\to\infty}\frac{b'(t)}{b^2(t)}
    & = - \frac1\mu \,\limsup_{t\to\infty} \frac\expo{(1+t)^{1-\expo}\,v(t)} \,,\\
\lim_{t\to\infty} \frac{t\,b'(t)}{b(t)}
    & = -\expo\,.
\end{align*}
Assumption~\eqref{eq:b1blow} is immediately satisfied, and assumption~\eqref{eq:liminfblow} trivially holds if~$v(t)$ also verifies
\begin{equation}\label{eq:vpert2}
\lim_{t\to\infty} (1+t)^{1-\expo}\,v(t) = +\infty\,.
\end{equation}
In particular, conditions~\eqref{eq:vpert} and~\eqref{eq:vpert2} hold if
\[
v(t)=(\log (e+t))^\gamma
\] for any~$\gamma\in\R$. More in general, if it is an iteration of logarithmic functions, possibly with different powers, like as:
\begin{align*}
v(t)
    & = (\log (e+ (\log (e +t))^{\gamma_2}))^{\gamma_1}, \qquad \text{or} \\
v(t)
    & = (\log (e+ (\log (e + (\log (e+\ldots )))^{\gamma_3}))^{\gamma_2})^{\gamma_1}.
\end{align*}
Conditions~\eqref{eq:vpert} and~\eqref{eq:vpert2} hold if $v(t)=w((1+t)^{-\alpha})$ for some~$\alpha>0$, where $w: [0,1]\to (0,\infty)$ is a strictly positive, $\mathcal{C}^1$ function. Indeed, $w$, $1/w$ and $w'$ are bounded functions, hence~\eqref{eq:vpert} and~\eqref{eq:vpert2} are satisfied. In this case, we still have $B(t)\approx t^{1+\expo}$. For instance,
\[ v(t) = 1+(1+t)^\alpha \sin(1+t)^{-2\alpha}
\]
obtained by taking $ w(r)=1+\frac{\sin r^2}r $ for $r\in(0,1]$ and $w(0)=1$.
\end{Example}


\section{Examples}\label{sec:Examples}

In the statements of Theorems \ref{Thm:main} and~\ref{Thm:mainCP}, for fixed $p>1$ we check our assumptions on a suitable $g$. One can pose the question on the structure of the set of exponents $p>1$ such that these assumptions hold. For a fixed~$g$, it is evident that if the assumptions hold for $p_1>1$ then they also holds for any $1<p<p_1$. Hence one may look for the largest interval $(1,\bar p)$ or $(1,\bar p]$ such that the assumptions remain valid. The literature sometimes refers to the exponent $\bar p$ as  critical Fujita exponent. In this Section, we will call these exponents \emph{Fujita-type}, whereas by Fujita exponent we only means~$p_\Fuj(n)\doteq 1+2/n$. Indeed, for semi-linear heat equation a nonexistence result for $1<p\le p_\Fuj(n)$ was provided by Fujita~\cite{fujita} (see Example~\ref{Ex:Fujita}). In the same paper one can find a global existence result for $p>p_\Fuj(n)$ and suitable initial data condition.

We are specially interested in Fujita-type exponents ~$\bar p$ for which it is known an existence result for $p>\bar p$. In such a case, we will say that they are \emph{critical}.

Let us remark that for the same equation one can find different exponents $\bar p$ for which the solution globally exists for~$p>\bar p$, according to which kind of solution one is interested in. The Fujita-type exponent is smaller than any possible existence exponent. In the critical case it coincides with an existence exponent.
\\
A typical example in this direction is given by the wave operator in dimension $n\ge 3$: the real numbers
$$
p_\Fuj(n-1)=1+\frac{2}{n-1}<p_\Str(n-1)=\frac{1}{2}+\frac{1}{n-1}+ \Big(\Big(\frac{1}{2}+\frac{1}{n-1}\Big)^2+\frac{2}{n-1}\Big)^{1/2}<p_\Sob(n)=\frac{n+2}{n-2}
$$
are respectively the Fujita-type exponent, the critical exponent for small amplitude solution, the large data critical exponent, see \cite{strauss}.
This says that for suitable small data and for $p_\Fuj(n-1)<p<p_\Str(n-1)$ pointwise solutions to $u_{tt}-\Delta u=|u|^p$ do not exist but weak solutions may ``survive''.
\\
A general discussion on the critical exponents can be found in~\cite{DeL, L}.

We first see how taking $g\equiv1$ in Theorems~\ref{Thm:main} and~\ref{Thm:mainCP} we find some already known nonexistence results for \emph{quasi-homogeneous} operators.
\begin{Example}\label{Ex:DL}
Theorem~\ref{Thm:main} easily applies to \emph{quasi-homogeneous} operators~$L$ such that~$L^*$ contains no zero order terms. In particular, following~\cite{DL}, let us write~$x=(x_1,x_2)$, with~$x_j\in\R^{n_j}$ and~$N_1+N_2=N$. We assume that there exists~$(d_1,d_2,h)$, with~$d_1,d_2>0$ and~$h\in\R$, such that the symbol of $L$ satisfies
\begin{equation}\label{eq:DL}
L (\lambda^{-d_1} x_1,\lambda^{-d_2} x_2, \lambda^{d_1}\xi_1,\lambda^{d_2}\xi_2) = \lambda^h L(x_1,x_2,\xi_1,\xi_2)\,,
\end{equation}
for any~$\lambda>0$. Then we set $F_i(R)=R^{d_1}$ for~$i=1,\ldots,N_1$ and~$F_i(R)=R^{d_2}$ for~$i=N_1+1,\ldots,N$. Using~\eqref{eq:DL} one can prove that
\begin{equation}\label{eq:DLscale}
L^* S_R = R^{-h} \, S_R\,L^* \,.
\end{equation}
If we write~$L^*$ as
\[ L^*=\sum_{1\le |\alpha|\le m}a_\alpha(x)\partial_x^\alpha\,,\]
then each operator~$a_\alpha(x)\partial_x^\alpha$ satisfies \eqref{eq:DL}-\eqref{eq:DLscale}. Therefore
\[ a_\alpha (x\cdot F(R)) = h_\alpha(R)\, a_\alpha(x)\,, \qquad \text{where}\quad h_\alpha(R)= R^{-h} \, \bigl(\prod_{i=1}^{N_1} R^{d_1\alpha_i} \bigr) \,\bigl(\prod_{i=N_1+1}^N R^{d_2\alpha_i} \bigr) \, \,. \]
Let us assume that~$f(x_1,x_2)\approx |x_1|^{\theta_1}\,|x_2|^{\theta_2}$ for some~$\theta_1,\theta_2\in\R$, and let~$g\equiv1$. Since
\begin{align}
\label{eq:GalphaDL}
G_\alpha(R)
	& \leq \det F(R)\,\int_{C_1} (h_\alpha(R)\, a_\alpha(x))^{p'}\,\bigl(|R^{d_1}\,x_1|^{\theta_1}|R^{d_2}\,x_2|^{\theta_2}\bigr)^{-(p'-1)}\,dx \,, \\
\label{eq:HalphaDL}
H_\alpha(R)
	& = \bigl(\prod_{i=1}^{N_1} R^{-d_1\alpha_i} \bigr) \,\bigl(\prod_{i=N_1+1}^N R^{-d_2\alpha_i} \bigr)\,,
\end{align}
we derive the estimate
\[ H_\alpha(R)\,(G_\alpha(R))^{\frac1{p'}} \lesssim R^{-h+\frac{(d_1\,N_1+d_2\,N_2)-(d_1\,\theta_1+d_2\,\theta_2)\,(p'-1)}{p'}} = R^{-(h+\theta)+\frac{d+\theta}{p'}} \,, \]
where we put~$\theta=d_1\,\theta_1+d_2\,\theta_2$ and~$d=d_1\,N_1+d_2\,N_2$, for any~$1\leq|\alpha|\leq m$. Applying Theorem~\ref{Thm:main} we obtain a nonexistence result for the Liouville problem~$Lu=f|u|^p$, if
\begin{equation}\label{eq:expDL}
p \leq \frac{d+\theta}{d-h} = 1 + \frac{\theta+h}{d-h}\,,
\end{equation}
provided that~$h\in(-\theta,d)$. The bound on the exponent~$p$ is the same obtained in~\cite{DL}. We remark that \eqref{eq:GalphaDL} implicitly requires $|x_1|^{-\theta_1}|x_2|^{-\theta_2} a_\alpha(x)\in L^{p'}_{\lloc}(|x_1|^{\theta_1}|x_2|^{\theta_2})$. This may give a lower bound for the exponent $p$.
\end{Example}
\begin{Example}\label{Ex:Fujita}
It is clear that the approach in Example~\ref{Ex:DL} holds in the setting of Cauchy problem for \emph{quasi-homogeneous} operators, provided that data satisfy~\eqref{eq:DATAcond}.

According to Theorem~\ref{Thm:mainCP}, for the semilinear heat equation
\[ u_t - \triangle u = |u|^p\,, \qquad u(0,x)=u_0(x)\,, \]
we find nonexistence of global solutions for any~$1<p\leq p_\Fuj(n)=1+2/n$, provided that~$u_0\in L^1$ and
\[ \int_{\R^n} u_0(x)\,dx > 0 \,. \]
Indeed, the heat operator is quasi-homogeneous of type (2,1,2) in $[0,\infty)\times \R^n$. 

If one considers the Cauchy problem for the semilinear Schr\"odinger equation
\[ Lu\equiv i u_t + \triangle u = |u|^p\,, \qquad u(0,x)=u_0(x)\,, \]
one finds again a nonexistence result for~$1<p\leq1+2/n$, provided that~$\Im u_0\in L^1$ and that
\[ \int_{\R^n} \Im u_0(x)\,dx < 0 \,. \]
We address the interested reader to~\cite{IW}, where the same result is extended to nonlinearities of type $\lambda |u|^p$, where~$\lambda$ is complex-valued. We remark that for $p>1+2/n$ the scattering theory of Schr\"odinger equation becomes meaningful, see \cite{TsYa}.
\\
The classical semilinear wave equation
\[ u_{tt} - \triangle u = |u|^p\,, \qquad u(0,x)=u_0(x) \quad u_t(0,x)=u_1(x)\,, \]
is also \emph{quasi-homogeneous} of type (2,1,1) and we find Kato exponent~ $p_\Kato(n)=1+2/(n-1)=p_\Fuj(n-1)$, provided that~$u_1\in L^1$ and that
\[ \int_{\R^n} u_1(x)\,dx > 0 \,. \]
We remark that Kato's result of nonexistence is obtained in~\cite{Kato} by comparison method, so that it holds for a larger class of operators.
\end{Example}
In some cases estimate~\eqref{eq:GalphaDL} can be refined to relax the restrictions on the exponent~$p$.
\begin{Example}\label{Ex:Tricomi}
Following Example~\ref{Ex:DL} and~\cite{DL}, one can find a nonexistence result for the semilinear equation for the Grushin operator
\[ \triangle_x u + g_\gamma(x) \triangle_y u = |x|^{\theta_1}\,|y|^{\theta_2}\,|u|^p\,,  \]
where $g_\gamma(x)$ is a homogeneous function of order $2\gamma$ with $\gamma\in\R$, $x\in\R^k$ for some~$1\leq k\leq N-1$ and~$y\in\R^{N-k}$. We are dealing with a quasi-homogeneous operator of type $(2,1,1+\gamma)$, so that the admissible range for $p$ is given by
\[ 1 + \max \left\{ \frac{[\theta_1]^+}k, \frac{[\theta_2]^+}{N-k} \right\} < p \leq 1 + \frac{2+\theta_1+(1+\gamma)\theta_2}{N+(N-k)\gamma-2} \,\]
if $(k+2\gamma)p>k+\theta_1$ and $(\gamma,\theta_1,\theta_2)$ satisfies $2+\theta_1+(1+\gamma)\theta_2>0$ and $N+(N-k)\gamma-2>0$.

In particular, for $k=1$ and $2\gamma\in \N^*$, we have the Tricomi Operator $\partial_{xx}  + x^{2\gamma}\triangle_y $. We consider the Liouville problem
\begin{equation}\label{eq:Tricomi}
u_{xx} + x^{2\gamma}\triangle_y u = |x|^\theta\,|u|^p\,,
\end{equation}
where~$x\in\R$ and~$y\in\R^{N-1}$, $N\ge 2$ and $\theta>-2$. Since $\max\{[\theta]^+, [\theta-2\gamma]^+\}=[\theta]^+$, one obtains a nonexistence result for
\begin{equation}\label{eq:tric}
1+[\theta]^+ < p \leq 1 + \frac{2+\theta}{N+(N-1)\gamma-2} \,.
\end{equation}
Let us see that the lower bound on~$p$ can be relaxed if we directly apply the estimate for~$G_\alpha$ given in~\eqref{eq:GalphaR}. Fixed $F(R)=(R^d,R,\ldots,R)$, for some~$d>0$,  we compute
\[ H_{\alpha} (R) = \begin{cases}
R^{-2d} & \text{if~$\alpha=2e_1$,}\\
R^{-2} & \text{if~$\alpha=2e_k$, \quad $k\geq2$.}
\end{cases} \]
Moreover, we can immediately estimate:
\begin{align}
\label{eq:G2e1}
G_{2e_1} (R)
    & \approx R^{d+N-1} \int_{1/2\leq |x|\leq1} (f(x\,R^d))^{-(p'-1)} dx \approx R^{d+N-1-d\,\theta(p'-1)} \,, \\
\label{eq:G2ek}
G_{2e_k} (R)
    & \approx R^{d+N-1} \int_{|x|\leq 1} |x\,R^d|^{2\gamma p'} (f(x\,R^d))^{-(p'-1)}\,dx \approx R^{d+N-1+2\gamma dp'-d\,\theta(p'-1)} \,,\quad \text{$k\geq2$,}
\end{align}
provided that~$2\gamma p'-(p'-1)\theta)>-1$. We remark that we do not have integrability problems in~\eqref{eq:G2e1} due to the fact that we integrate in~$1/2\leq|x|\leq1$. This approach represents an improvement with respect to the estimate~\eqref{eq:tric}.
Taking $d>0$ such that the same 
exponent appears in \eqref{eq:G2e1} and \eqref{eq:G2ek}, from Theorem~\ref{Thm:main}, we get
\begin{equation}\label{eq:pTricomi}
1 + \frac{[\theta-2\gamma]^+}{1+2\gamma} < p \leq 1 + \frac{\theta+2}{N+(N-1)\gamma-2} \,
\end{equation}
provided that $\theta>-2$. We remark that for some $\theta, N,\gamma$ the previous range for $p$ is empty. The same result holds for
\[
 u_{xx} + x^{2\gamma}\triangle_y u = f(x)\,|u|^p\,,
\]
for any $f$ such that $f(x)\gtrsim |x|^\theta$.
We remark that if~$f(x)\gtrsim \<x\>^\Theta$ for some~$\Theta>1$ in~\eqref{eq:Tricomi}, then assumption~$f(x)\gtrsim |x|^\theta$ is verified for any~$\theta\in [1,\Theta]$, hence the bound in~\eqref{eq:pTricomi} is replaced by
\begin{equation}\label{eq:pTricomi2}
1 < p \leq 1 + \frac{\theta+2}{N+(N-1)\gamma-2} \,.
\end{equation}
\end{Example}
Condition~\eqref{eq:pTricomi} improves from below the range of admissible expoenents~$p$ for Grushin operators with nonlinearities~$|x|^\theta$ with respect to~\cite{DL}. Let us come back to the original question posed by Deng and Levin in~\textsection 5 in~\cite{DeL}. Taking a nonlinearity in the form
\[ f(x)|u|^p = |x^{(1)}|^{\theta_1}\,\ldots\,|x^{(k)}|^{\theta_k} |u|^p \, \]
where~$x^{(j)}\in\R^{N_j}$ and~$N_1+\ldots+N_k=N$, we may obtain more admissible exponents for a quasi-homogeneous operator with constant coefficients, with respect to the approach in~\cite{DL}. This improvement is a consequence of the choice of using nonradial test functions in the form~\eqref{eq:psiprod}. Indeed, in this way, for any~$\alpha\in\N^N$ such that~$\partial_x^\alpha x_k^{(j)}=0$ for some~$j$ and for any~$k=1,\ldots,N_j$, the domain of~$G_\alpha$ absorbs potential singularities coming from~$|x^{(j)}|^{\theta_j}$ if~$\theta_j>0$.
\begin{Example}
We notice that no sign assumption on the coefficients plays a particular role in Example~\ref{Ex:Tricomi}, therefore the result of nonexistence is still valid for $u_{xx} -x^{2\gamma}\triangle_y u = |x|^\theta\,|u|^p$. As a consequence we may study the Cauchy problem
\begin{equation}\label{eq:wave}
\begin{cases}
u_{tt} - t^{ 2\gamma}\triangle u = f(t)\,|u|^p\,\quad t\in [0,\infty) \text{ and } x\in\R^n\\
u(0,x)=u_0(x)\,,\\
u_t(0,x)=u_1(x)\,,
\end{cases}
\end{equation}
where $f(t)$ is a positive function satisfying~$f(t)\gtrsim t^\theta$ for some~$\theta>-2$.
If~$u_1\in L^1$ and
\[ \int_{\R^n} u_1(x)\,dx > 0 \,, \]
then there exists no global weak solution to~\eqref{eq:wave} if 
\begin{equation}\label{eq:pTriwave}
1 + \frac{[\theta-2\gamma]^+}{1+2\gamma} < p \leq 1 +\frac{\theta+2}{n+n\gamma-1} \,.
\end{equation}
For $\gamma=0$ we find again Kato exponent~$p_\Kato(n)$.
\end{Example}
Even if an operator is quasi-homogeneous, it may happen that~$L^*$ contains a term of order zero. In this case, Theorems \ref{Thm:main} and \ref{Thm:mainCP} come into play if one finds a function~$g$ such that~$D^*=(gL)^*$ does not contain zero order terms. In the next example $D^*$ will also remain quasi-homogeneous.
\begin{Example}\label{Ex:HL} 
Let us consider the Cauchy problem
\begin{equation}\label{eq:HL} \begin{cases}
Lu\equiv \partial_t^m u -\triangle u + \frac\lambda{|x|^2} =|u|^p \,\quad t\ge 0, \, x\in \R^n\\
\partial_t^k u(0,x)=u_k(x)\,,\qquad k=0,\ldots,m-1\,,
\end{cases}
\end{equation}
where~$\lambda\ge 0$, in space dimension~$n\geq3$. Assuming~$u_{m-1}(x)\geq0$ a.e. and defining
\[ s=s(n,\lambda) \doteq \sqrt{\frac{(n-2)^2}4+\lambda}\, - \frac{n-2}2 \,, \]
in~\cite{HL} it was proved that there exists no global solution to~\eqref{eq:HL} for
\begin{equation}\label{eq:HLcrit}
1 < p \leq 1 + \frac2{n-2 + s + 2/m} \,.
\end{equation}
According to Remarks~\ref{Rem:LucPez} and~\ref{Rem:pointpos}, we may obtain the same result using Theorem~\ref{Thm:mainCP}, provided that~$s>1$, that is, $\lambda>n-1$. Indeed, if~$u$ is a $g$-solution, we may set~$F_0(R)=R^{\frac2m}$, $F_i(R)=R$ for~$i=1,\ldots,n$, and~$g(x)=|x|^{s}$. Therefore
\begin{equation}\label{eq:HLg}
-\triangle g + \frac\lambda{|x|^2} g = 0\,,
\end{equation}
so that
\[ D^* = g(x) (-1)^m \partial_t^m -g(x)\triangle -\nabla g(x) \cdot \nabla \]
does not contain terms of order zero. We notice that the assumption $s>1$ comes into play to guarantee that~$D^*$ has $L^\infty_\lloc$ coefficients. Moreover, we may replace the assumption~$u_{m-1}(x)\geq0$ by taking $u_{m-1}\in L^1(g(x)dx)$ and
\[ \int_{\R^n} g(x) u_{m-1}(x)\,dx > 0\,. \]
Stretching a little bit Theorem~\ref{Thm:mainCP}, we may also consider the case~$\lambda\in[0,n-1]$, in particular if we assume $u_{m-1}(x)\geq0$ and we want to prove the non existence of \emph{weak solution}.
\end{Example}
\begin{Rem}\label{Rem:Katomod} 
We remark that for~$m=2$ in~\eqref{eq:HLcrit} one finds
\[ 1 < p \leq 1 + \frac2{n-1+s} \,. \]
The upper bound gives Kato exponent~$p_\Kato(n)=1+2/(n-1)$ if~$\lambda=0$, and it can be considered like a Kato exponent modified by the influence of the mass term if~$\lambda>0$. We remark that~$s\geq0$, so that it is worse than~$p_\Kato(n)$. Similarly for $m=1$ the relation \eqref{eq:HLcrit} becomes
\[ 1 < p \leq 1 + \frac2{n+s} \,, \]
which is a Fujita modified exponent. Moreover in~\cite{HL} it is shown that this exponent is critical, that is, there exist global positive solutions to~\eqref{eq:HL} for~$m=1$.
\end{Rem}
If we add a mass term in the form~$a(t)\,\lambda/|x|^2$ to the damped wave equation in~\eqref{eq:dissblow}, then we can use again a \emph{modified} test function method. Moreover, in this case we are no more dealing with \emph{quasi-homogeneous} operators.
\begin{Example}\label{Ex:HLD} 
Let~$b(t)$ be as in Hypothesis~\ref{Hyp:blowup} and let us consider the Cauchy problem
\begin{equation}
\label{eq:dissblowHL}
\begin{cases}
u_{tt}-a(t)\triangle u+b(t)u_t + a(t)\,\frac\lambda{|x|^2} u =f(t,x)\,|u|^p, & t\geq0, \ x\in\R^n,\\
u(0,x)=u_0(x), \\
u_t(0,x)=u_1(x).
\end{cases}
\end{equation}
where~$n\geq3$ and $\lambda>n-1$. Now we may apply Theorem~\ref{Thm:mainCP} with~$g=g_1(t)\,g_2(x)$ where~$g_1(t)=\Gamma(t)/\beta(t)$ as in the proof of Theorem~\ref{Thm:blowup}, and~$g_2(x)=|x|^{s}$ as in Example~\ref{Ex:HL}.
\\
By virtue of~\eqref{eq:geq} and~\eqref{eq:HLg}, we have
\[ D^* = g(t,x)\partial_t^2 - g(t,x)\,a(t)\triangle + g_2(x)\,(g_1'(t)-1)\partial_t - a(t)\,g_1(t)\nabla g_2(x)\cdot \nabla \,. \]
We claim that if
\begin{itemize}
\item $0<a(t)\lesssim B(t)^{-\alpha}$ for some~$\alpha<1$,
\item $f(t,x)\gtrsim B(t)^\gamma\,|x|^\delta$ for some $\gamma>-1$ and~$\delta\in\R$,
\end{itemize}
as in Theorem~\ref{Thm:blowup}, then there exists no global solution for
\begin{equation}\label{eq:HLDcrit}
1 + \max \left\{ \frac{[\gamma+\alpha]^+}{1-\alpha} , \frac{[\delta]^+}{n+s} \right\} <p \leq 1 + \frac2{n+s}\,\left(\frac{1+\gamma}{1-\alpha} + \frac\delta2\right)\,,
\end{equation}
provided that~$u_0,u_1\in L^1(g_2(x)dx)$ and
%
\[ \int_{\R^n} g_2(x)\,\big(u_1(x)+\hat{b}_1u_0(x)\big) dx > 0\,.\]
%
Indeed, it is sufficient to follow the proof of Theorem~\ref{Thm:blowup}, replacing
\[ \int_{Q_R} |x|^{-\delta(p'-1)}\,dx \qquad \text{with} \qquad \int_{Q_R} |x|^{s-\delta(p'-1)}\,dx \]
in the estimates for~$G_{2e_0}, G_{e_0}$ and $G_{2e_i}$ for~$i=1,\ldots,n$. Moreover, we have to consider the new terms
\[ G_{e_i}(R) \lesssim \int_0^{A(R^d)} g_1(t)\,B(t)^{-\gamma(p'-1)-\alpha p'} \int_{Q_R\setminus Q_{R/2}} |x|^{s -p'-\delta(p'-1)}\,dx dt\,, \]
for any~$i=1,\ldots,n$. Since~$s-\delta(p'-1)>1$ thanks to~\eqref{eq:HLDcrit}, setting again~$d=2/(1-\alpha)$, we may then arrive at the estimate
\[ H_{\tilde{\alpha}}(R)\,(G_{\tilde{\alpha}}(R))^\frac1{p'} \lesssim R^{\frac1{p'}(n+s) -\frac1p(\frac{1+\gamma}{1-\alpha}+\delta)} \,, \]
for any~$\tilde{\alpha}=2e_i, e_i$, for~$i=0,\ldots,n$. Therefore we can apply Theorem~\ref{Thm:mainCP}, proving our claim.
\\
If~$\gamma=-\alpha$, as in Example~\ref{Ex:special} and~$\delta>-2$, then~\eqref{eq:HLDcrit} becomes:
\[ 1 + \frac{[\delta]^+}{n+s} <p \leq 1 + \frac{2+\delta}{n+s}\,.\]
For~$\delta=0$ we have nonexistence for ~$1<p\leq 1+2/(n+s)$, the Fujita exponent modified by the influence of the mass term, as in Remark~\ref{Rem:Katomod}.
\end{Example}

\begin{Rem}
An analogous result can be obtained for the Liouville problem associated to the equation
$$
\partial_{yy}\pm a(y)\triangle u+b(y)u_y \pm a(t)\,\frac\lambda{|x|^2} u =f(y,x)\,|u|^p\,, \quad y\in\R\,,\quad x\in\R^{N-1}\,,
$$
with 
$a,b,f,\lambda$ satisfying the same assumptions of Example \ref{Ex:HLD}. One obtains a nonexistence result for any $p$ in the range \eqref{eq:HLDcrit}, by replacing~$N=n+1$ as usual.
\end{Rem}



\section*{Acknowledgments}

The authors thank Yuta Wakasugi for some early discussions on this topic.


\begin{thebibliography}{99}

\bibitem{DA12}
{\sc M. D'Abbicco}
{\it Small data solutions for semilinear wave equations with effective damping},
submitted to Proceedings of $9^{\mathrm{th}}$ AIMS Conference, Orlando, Florida, 2012, 9 pages.

\bibitem{DA12+}
{\sc M. D'Abbicco}
{\it Semilinear scale-invariant wave equations with time-dependent speed and damping}, preprint, 20 pages.

\bibitem{DAE12}
{\sc M. D'Abbicco, M.R. Ebert}, {\it Hyperbolic-like estimates for higher order equations}, J. Math. Anal. Appl. {\bf 395} (2012) 747--765. 

\bibitem{DAE13}
{\sc M. D'Abbicco, M.R. Ebert}, {\it A class of dissipative wave equations with time-dependent speed and damping}, J. Math. Anal. Appl. (2012), doi:10.1016/j.jmaa.2012.10.017.

\bibitem{DLR}
{\sc M. D'Abbicco, S. Lucente, M. Reissig,}
{\it Semilinear wave equations with effective damping},
\rm ArXiv: 1210.3493v1.

\bibitem{DL}
{\sc L. D'Ambrosio, S. Lucente,} {\it Nonlinear Liouville theorems
for Grushin and Tricomi operators,} J. Differential Equations, {\bf 123}
(2003), 511--541.

\bibitem{D}
\sc P. D'Ancona, \rm
{\it A Note on a Theorem of J\"orgens},
Math. Z., {\bf 218} (1995), 239--252.

\bibitem{DD}
\sc P. D'Ancona, A. Di Giuseppe, \rm
{\it Global Existence with Large Data for a Nonlinear Weakly Hyperbolic Equation},
Math. Nachr., {\bf 231} (2001), 5--23.

\bibitem{DeL}
\sc K. Deng and H.A. Levine. \it The role of critical exponents in blow-up theorems: the sequel,
\rm J. Math. Anal. Appl., \bf 243 \rm (2000), 85--126.

\bibitem{HL}
{\sc A. El Hamidi, G. Laptev},
\it Existence and nonexistence results for higher-order semilinear evolution inequalities with critical potential.
\rm J. Math. Anal. Appl., \bf 304 \rm (2005), 451--463.

\bibitem{Fanelli}
{\sc L. Fanelli}, {\it
Semilinear Schrödinger equation with time dependent coefficients},
Math. Nachr.  {\bf 282}, (2009) 976--994.

\bibitem{FL}
{\sc L. Fanelli, S. Lucente}, {\it The critical case for a semilinear weakly hyperbolic equation},
El. Journ. of Diff. Eq. {\bf 2004} No. 101 (2004), 1--13.

\bibitem{fujita}
{\sc H. Fujita},{\it On the blowing up of solutions of the Cauchy problem for $u_t=\Delta u+u^{1+\alpha}$},
  J. Fac. Sci. Univ. Tokyo, Sec. I, {\bf 13} (1966), 109--124.

\bibitem{IW}
{\sc M. Ikeda, Y. Wakasugi},
{\it Small data blow-up of $L^2$ solution for the nonlinear Schr\"odinger equation without gauge invariance}, ArXiv:1111.0178v3.

\bibitem{ITY}
{\sc R. Ikehata, G. Todorova, B. Yordanov},
{\it Critical exponent for semilinear wave equations with space-dependent potential},
Funkcialaj Ekvacioj, \textbf{52} (2009), 411--435.

\bibitem{Kato}
\sc T. Kato, \it Blow-up of solutions of some nonlinear hyperbolic equations,
\rm Commun. Pure Appl. Math., \bf 33 \rm (1980), 501--505.

\bibitem{L}
\sc H. A. Levine, \it The role of critical exponents in blow-up problems,
\rm SIAM Review, \bf 32 \rm (1990), 262--288.

\bibitem{LNZ11}
{\sc J. Lin, K. Nishihara, J. Zhai}
\it Decay property of solutions for damped wave equations with space-time dependent damping term,
\rm J. Math. Anal. Appl. \bf 374 \rm (2011), 602--614

\bibitem{LNZ}
{\sc J. Lin, K. Nishihara, J. Zhai}, \emph{Critical exponent for the semilinear wave equation with
time-dependent damping}, Discrete and Continuous Dynamical Systems, \bf 32 \rm (2012), 4307--4320.

\bibitem{Lu}
\sc S. Lucente, \it On a class of semilinear weakly hyperbolic equations,
\rm Annali dell'Universita di Ferrara \bf 52 \rm (2006), 317--335.

\bibitem{MP1}
\sc E. Mitidieri, S. I. Pohozaev,
\it The absence of Global Positive Solutions
to Quasilinear Elliptic Inequalities, \rm Doklady Mathematics, \bf 57 \rm (1998),
250--253.

\bibitem{MP2}
\sc E. Mitidieri, S. I. Pohozaev,
\it  Nonexistence of Positive Solutions for a
Systems of Quasilinear Elliptic Equations and Inequalities in $\R^n$, \rm Doklady
Mathematics , \bf 59 \rm (1999), 1351--1355.

\bibitem{MP}
{\sc E. Mitidieri, S.I. Pohozaev},
{\it Non-existence of Weak Solutions for some Degenerate Elliptic and Parabolic Problems on $\R^n$,}
Journal of Evolution Equations, {\bf 1} (2001), 189--220.

\bibitem{MPhyp}
{\sc E. Mitidieri, S.I. Pohozaev},
{\it Nonexistence of Weak Solutions for Some Degenerate and Singular Hyperbolic Problems on $\R^n$},
Proc. Steklov Institute of Mathematics {\bf 232} (2001), 240--259.

\bibitem{R}
{\sc M. Reissig}, \emph{Klein-Gordon type decay rates for wave equations with a time-dependent dissipation}, Adv. Math. Sci. Appl. \bf 11 \rm (2001), 859--891.

\bibitem{TsYa}
\sc Y. Tsutsumi, K. Yajima, \it The asymptotic behavior of nonlinear Schr\"odinger equations, \rm
Bull. Amer. Math. Soc., \bf 11 \rm (1984), 186--188.

\bibitem{strauss}
{\sc W.A. Strauss}, \it  Nonlinear wave equations, \rm CBMS Regional Conference Series in
Mathematics, {\bf 73}, Amer. Math. Soc. Providence, RI, 1989.

\bibitem{TY}
{\sc G. Todorova, B. Yordanov }, \emph{Critical Exponent for a Nonlinear Wave Equation with Damping},
J. Differential Equations, \bf 174 \rm (2001), 464--489.

\bibitem{W}
{\sc Y. Wakasugi}, \it Small data global existence for the semilinear wave equation with space-time dependent damping,
\rm ArXiv:1202.5379v1

\bibitem{W07}
{\sc J. Wirth}, \it Wave equations with time-dependent dissipation II. Effective dissipation,
\rm J. Differential Equations, \bf 232 \rm (2007), 74--103.

\bibitem{zhang}
{\sc Q. S. Zhang}, \emph{A blow-up result for a nonlinear wave equation with damping: the critical case},
C. R. Acad. Sci. Paris S\'{e}r. I Math., \bf 333 \rm (2001), 109--114.

\end{thebibliography}
\end{document}